\documentclass[9pt,shortpaper,twoside,web]{ieeecolor}
\usepackage{generic}
\usepackage{graphics} % for pdf, bitmapped graphics files
\usepackage{epsfig} % for postscript graphics files
\usepackage{epstopdf} % for postscript graphics files
\usepackage{times} % assumes new font selection scheme installed
\usepackage{amsmath} % assumes amsmath package installed
\usepackage{amssymb,amsfonts}  % assumes amsmath package installed
\usepackage{subcaption}
\usepackage{float}
\usepackage{algorithm}
\usepackage{algpseudocode}
\usepackage{xspace}
\usepackage{verbatim}
\usepackage{multirow}
\usepackage{csquotes}
\usepackage{booktabs}
\usepackage{cite}
\usepackage{textcomp}
\usepackage[hidelinks]{hyperref}

\usepackage{textcomp}
\def\BibTeX{{\rm B\kern-.05em{\sc i\kern-.025em b}\kern-.08em
		T\kern-.1667em\lower.7ex\hbox{E}\kern-.125emX}}
\newtheorem{theorem}{Theorem}
\newtheorem{lemma}{Lemma}
\newtheorem{corollary}{Corollary}
\newtheorem{remark}{Remark}
\newtheorem{defn}{Definition}
\newtheorem{assume}{Assumption}

\newcommand{\algrulehor}[1][.2pt]{\par\vskip.5\baselineskip\hrule height #1\par\vskip.5\baselineskip}

\begin{document}

\title{\LARGE \bf
A Sensitivity-based Data Augmentation Framework for Model Predictive Control Policy Approximation
}

\author{Dinesh Krishnamoorthy \IEEEmembership{Member, IEEE}% <-this % stops a space
	\thanks{*This work was  supported by the Research Council of Norway, under the IKTPLUSS program (Project number 299585)}% <-this % stops a space
	\thanks{Dinesh Krishnamoorthy is currently with the Harvard John A. Paulson School of Engineering and Applied Sciences, Harvard University, Cambridge, MA, 02138.
		(email: {\tt\small dkrishnamoorthy@seas.harvard.edu})}%
}

\maketitle

\begin{abstract}
	Approximating model predictive control (MPC) policy using expert-based supervised learning techniques  requires labeled training data sets sampled from the MPC policy. This is typically obtained by sampling the feasible state-space and evaluating the control law by solving the numerical optimization problem offline for each sample. Although the resulting approximate policy can be cheaply evaluated online, generating large training samples to learn the MPC  policy can be time consuming and prohibitively expensive. This is one of the fundamental bottlenecks that limit the design and implementation of  MPC policy approximation. This technical note aims to address this challenge, and proposes a novel sensitivity-based data augmentation scheme for direct policy approximation. The proposed approach is based on exploiting the parametric sensitivities to cheaply generate additional training samples in the neighborhood of the existing samples. 
\end{abstract}

%%%%%%%%%%%%%%%%%%%%%%%%%%%%%%%%%%%%%%%%%%%%%%%%%%%%%%%%%%%%%%%%%%%%%%%%%%%%%%%%
\section{INTRODUCTION}
Model predictive control (MPC) is a popular control strategy  for constrained multivariable systems that is based on repeatedly solving a receding horizon optimal control problem at each sampling time of the controller.  As the range of MPC application extends beyond the traditional process industries, additional challenges such as computational effort and memory footprint need to be addressed. One approach to eliminate the need for solving optimization problems online, is to pre-compute the MPC policy $ u^* = \pi(x)$ as a function of the states $ x$. 
This idea was first proposed under the context of \emph{explicit MPC} for constrained linear quadratic systems where the MPC feedback law is expressed as a piecewise-affine function defined on polytopes \cite{bemporad2002explicit,tondel2003algorithm}. However, this can quickly become computationally intractable for large systems, since the number of polytopic regions grow exponentially with the number of decision variables and constraints. The extension to nonlinear systems (with economic objective terms) is also not straightforward.

An alternative approach is to use some parametric function approximator, such as artificial neural networks (ANN) to approximate the MPC policy. Although this idea dates back to the mid 90s \cite{parisini1995receding}, the use of parametric functions to approximate the MPC policy remained more or less dormant until very recently. 
Motivated by the recent developments and promises of deep learning techniques, there has been an unprecedented surge of interest in the past couple of years in approximating the MPC policy using deep neural networks. This interest has resulted in a number of research works from several research groups published just in the past couple of years. 
See for example \cite{chakrabarty2017support,karg2018efficient,chen2018approximating,csekHo2015explicit,paulson2020approximate,hertneck2018learning,drgovna2018approximate,zhang2019safe} to name a few.

The underlying  framework adopted in these works is as follows. The feasible state-space is sampled offline to generate a finite number of $ N_{s} $ discrete states $ \{ x_{i} \}_{i=1}^{N_{s}} $. The NMPC problem is solved offline for each discrete state as the initial condition to obtain the corresponding optimal control law $ u_{i}^*  = \pi_{\textup{mpc}}(x_{i})$ for all $ i = 1,\dots,N_{s} $. The resulting MPC policy $ \pi_{\textup{mpc}}(\cdot)  $ is approximated using any suitable parametric function approximator with $ \{(x_{i},u_{i}^*)\}_{i=1}^{N_{s}} $ as the labeled training data set, such that the trained policy $ \pi_{\textup{approx}}(\cdot) $ can be used online to cheaply evaluate the optimal control input. This approach is also studied more broadly under the context of \textit{direct policy approximation}, and is also known as {expert-based supervised learning} \cite{bertsekas2019reinforcement}, behavioral cloning \cite{esmaili1995behavioural}, or imitation learning \cite{osa2018algorithmicperspectiveImitationLearning,novak2020supervised}, where in our case the MPC policy is the \enquote{expert} that we would like to imitate as faithfully as possible.

One of the main bottlenecks of this approach is that, generating the training data set (given by the \enquote{expert}) can be time consuming and prohibitively expensive. The availability of large training data set covering the entire feasible state space is a key stipulation in using deep learning techniques and has a major impact on  the accuracy of the approximate policy.  This implies that the sample size $ N_{s} $ must be sufficiently large, covering the entire feasible state space. In the case of MPC policy approximation, one  then typically has to solve a large number of nonlinear programming (NLP) problems offline in order to generate adequate training samples. This challenge is only amplified for higher dimensional systems, since the number of samples $ N_{s} $ required to adequately cover the feasible state-space increases exponentially with the number of states. 
For example, the authors in \cite{hertneck2018learning} reported a computation time of roughly 500 hours on a Quad-Core PC\footnote{without parallelization of the sampling and validation} to learn the approximate MPC control law for the case study considered in their work.  Other works also report the need for a large training data set  to adequately approximate the MPC policy.  

In the field of machine learning and deep neural networks, the problem of insufficient training data samples is typically addressed using a process known as \enquote{data augmentation}, which is \emph{ a strategy to artificially increase the number of training samples using computationally inexpensive transformations of the existing samples} \cite{tran2017bayesian,taylor2018improving,shorten2019survey}. This has been extensively studied in the context of deep learning for image classification problems, where geometric transformations (such as translation, rotation, cropping etc.) and photometric transformations (such as color, saturation, contrast, brightness etc.) are often used to augment the existing data set with artificially generated training samples. Unfortunately, such data augmentation techniques are not relevant when the training data set is obtained by sampling some expert policy, and hence  are not applicable in the context of control policy approximation. 

This technical note aims to address the key issue of generating the training data samples by exploiting the NLP sensitivities to generate multiple training samples using the solution of a single optimization problem solved offline. That is, the MPC problem solved offline  can be seen as a parametric optimization problem parameterized with respect to the initial state $ x_{i} $.  The NLP sensitivity then  tells us how the optimal input $ u^*_{i} $ changes for perturbations $ \Delta x $ in the neighborhood of $ x_{i} $. 
Therefore, using  the solution to one parametric optimization problem solved for $ x_{i} $, we can cheaply generate multiple training data samples for other state realizations  in the neighborhood of $ x_{i} $ using the NLP sensitivity (also known as tangential predictor). This only requires computing the solution to a system of linear equations, which is much cheaper to evaluate than solving a nonlinear programming problem. To this end, the aim of this technical note is not to present a new MPC approximation algorithm, but rather address the  pivotal issue of generating training data samples, that would facilitate  efficient design and  implementation of the  approximate  MPC framework. 

\paragraph*{Main contribution} The main contribution of this technical note is a sensitivity-based data augmentation framework to efficiently and cheaply generate training data samples that can be used to approximate the MPC policy. 
More broadly, the proposed scheme is not just restricted to data augmentation for approximating an MPC policy, but can also be used in cases where the training data set is generated by solving an NLP problem e.g.  in inverse optimal control (where the goal is to impute the objective from data sampled from an optimal policy \cite{kalman1964InverseOCP,keshavarz2011imputing}), approximate moving horizon estimation \cite{karg2021approximateMHE},  steady-state real-time optimization \cite{DK2019FOPAM} etc. Thus, a wider contribution of this technical note is a  novel data augmentation framework for approximate optimal control problems, where the training data comprises of optimal state-action pairs, generated by sampling some  policy given by a nonlinear programming problem. 

The remainder of the paper is organized as follows. Section~\ref{sec:preliminaries} formulates the problem and recalls the MPC policy approximation framework. The sensitivity-based data augmentation technique to efficiently generate the training samples is presented in Section~\ref{sec:Proposed}, where we also provide an upper bound on the approximation error stemming from augmenting the data set with inexact samples. The proposed approach is illustrated using two different examples in Section~\ref{sec:Example} before concluding the paper in Section~\ref{sec:Conclude}.

\section{PRELIMINARIES} \label{sec:preliminaries}
\subsection{Problem Formulation}
Consider a discrete-time nonlinear system,
\begin{equation}\label{Eq:System}
x(t+1)=f\left(x(t),u(t) \right)
\end{equation}
where $ x(t) \in \mathbb{R}^{n_{x}} $ and $ u(t) \in \mathbb{R}^{n_{u}} $ are the states and  control inputs at time $ t $, respectively. The mapping $ f:\mathbb{R}^{n_{x}}\times\mathbb{R}^{n_{u}} \rightarrow \mathbb{R}^{n_{x}} $ denotes the discrete time plant model.
The MPC problem $ \mathcal{P}(x(t)) $ is formulated as
\begin{subequations}\label{Eq:MPC}
\begin{align}
 V_{N}(x(t)) =&\min_{x(\cdot|t),u(\cdot|t)}  \; \sum_{k=0}^{N-1} \ell(x(k|t)u(k|t)) + \ell_f(x(N|t))\label{Eq:MPC:cost}  \\
\textup{s.t.} \; & x({k+1}|t) = f(x(k|t),u(k|t)) \quad\forall k \in \mathbb{I}_{0:N-1}\label{Eq:MPC:system}\\
& x(k|t) \in \mathcal{X}, \quad u(k|t) \in \mathcal{U} \qquad\quad\forall k \in \mathbb{I}_{0:N-1}\label{Eq:MPC:path}\\
& x({N|t}) \in \mathcal{X}_{f} \label{Eq:MPC:terminal}\\
& x({0|t}) = x(t) \label{Eq:MPC:init}
\end{align}
\end{subequations}
where $ \ell: \mathbb{R}^{n_{x}} \times \mathbb{R}^{n_{u}} \rightarrow \mathbb{R}$ denotes the stage cost, which may be either a tracking or economic objective,  $ \ell_f : \mathbb{R}^{n_{x}} \rightarrow \mathbb{R}$ denotes the terminal cost, $ N $ is the length of the prediction horizon,  \eqref{Eq:MPC:path} denotes the state and input constraints, \eqref{Eq:MPC:terminal} denotes the terminal constraint, and \eqref{Eq:MPC:init} denotes the initial condition constraint. In the traditional MPC paradigm, the optimization problem $ \mathcal{P}(x(t)) $ is solved at each sample time $ t $ using $ x(t) $ as the state feedback, and the optimal input $  u^*(t) = u^*(0|t) $ is injected into the plant in a receding horizon fashion. This implicitly leads to the control policy $ \pi_{\textup{mpc}}: \mathbb{R}^{n_{x}} \rightarrow\mathbb{R}^{n_{u}} $.
\begin{equation}\label{Eq:MPClaw}
u^*(t) = \pi_{\textup{mpc}}(x(t)) 
\end{equation}

\subsection{MPC policy approximation}
This subsection recalls the underlying idea of the  MPC policy approximation framework common to works such as \cite{parisini1995receding,karg2018efficient,paulson2020approximate} and \cite{hertneck2018learning}. To approximate the MPC control law \eqref{Eq:MPClaw}, the feasible state space $ \mathcal{X}_{feas} $ is sampled to generate $ N_{s} $ randomly chosen initial states $ \{x_{i}\}_{i=1}^{N_{s}}$. For each initial state $ x_{i} $, the MPC problem $ \mathcal{P}(x_{i}) $ is solved to obtain the corresponding optimal input $ u^*_{i}  = \pi_{\textup{mpc}}(x_{i})$.  Using the data samples $ \mathcal{D} := \{ (x_{i},u^*_{i})\}_{i=1}^{N_{s}} $, any desirable parametric function $ \pi_{\textup{approx}}(x;\theta) $ parameterized by the parameters $ \theta $ is trained that minimizes the mean squared error
\begin{equation}\label{Eq:MSE}
 {\theta}_{0} = \arg \min_{\theta} \frac{1}{N_{s}}\sum_{i=1}^{N_{s}} \| \pi_{\textup{approx}}(x_{i};\theta) - u^*_{i} \|^2 
\end{equation}
This is summarized in Algorithm~\ref{alg:Training}.  Once the parametric function is trained, the approximate control policy $ \pi_{\textup{approx}}(x;\theta_{0})  $ can be used online to cheaply evaluate the optimal control input. %, and the closed loop system under the approximate MPC control law is given by,
Although deep neural networks have become a popular choice for MPC policy approximation, several other function approximators have also been used in the literature. As such, the method proposed in the following section will not be limited to any one class  of parametric functions. 

\begin{algorithm}[t]
	\caption{Generating training samples and approximating the MPC policy.}
	\label{alg:Training}
	\begin{algorithmic}[1]
		\Require $ \mathcal{P}(x) $, $ \mathcal{X}_{feas} $, $ \mathcal{D} = \emptyset$
		\algrulehor
		\For {$ i = 1,\dots,N_{s} $}
		\State Sample $ x_{i} \in \mathcal{X}_{feas} $
		\State $ u_{i}^* \leftarrow $ Solve $ \mathcal{P}(x_{i}) $
		\State $\mathcal{D} \leftarrow \mathcal{D} \cup \{(x_{i},u_{i}^*)\}$
		\EndFor
		\State $ {\theta}_{0} \leftarrow\arg \min_{\theta} \frac{1}{N_{s}}\sum_{i=1}^{N_{s}} \| \pi_{\textup{approx}}(x_{i};\theta) - u^*_{i} \|^2  $
		\algrulehor
		\Ensure $ \pi_{\textup{approx}}(x;{\theta}_{0}) $
	\end{algorithmic}
\end{algorithm}

\section{Motivating example} To better illustrate the idea of augmenting inexact samples, consider the simple pathological example where we want to approximate a one dimensional policy function  $ \pi: \mathbb{R} \rightarrow \mathbb{R}$ with a parametric function $ \pi(x,\theta) $, which is chosen  to be a  $ 7^{th} $ order polynomial. 
The optimal policy $ \pi^*(x) $ that we would like to approximate  is shown in Fig.~\ref{Fig:pathologicalEx} (in blue). 

To approximate the optimal policy with a polynomial, $ N=4 $ data points are queried 
from the expert (shown in red circles). This data set is denoted by $\mathcal{D}^0$. We compute the coefficients of the polynomial to fit $\mathcal{D}^0$ in a least squares sense, and the fitted polynomial $ \pi _{}(x,\theta_0) $ is shown in Fig.~\ref{Fig:pathologicalEx} (left subplot in gray). Although the fitted polynomial matches the queried data points exactly, it clearly does not approximate the policy. %and clearly cannot be used to approximate the optimal policy. 
Note that this is a common issue with over-parameterized functions and too few training data samples, which can easily occur for example when using deep neural networks. This simple pathological example clearly motivates the need for augmenting the data set with more samples in order to improve generalization to the states not queried from the expert.

\begin{figure*}
	\centering
	\includegraphics[width=\linewidth]{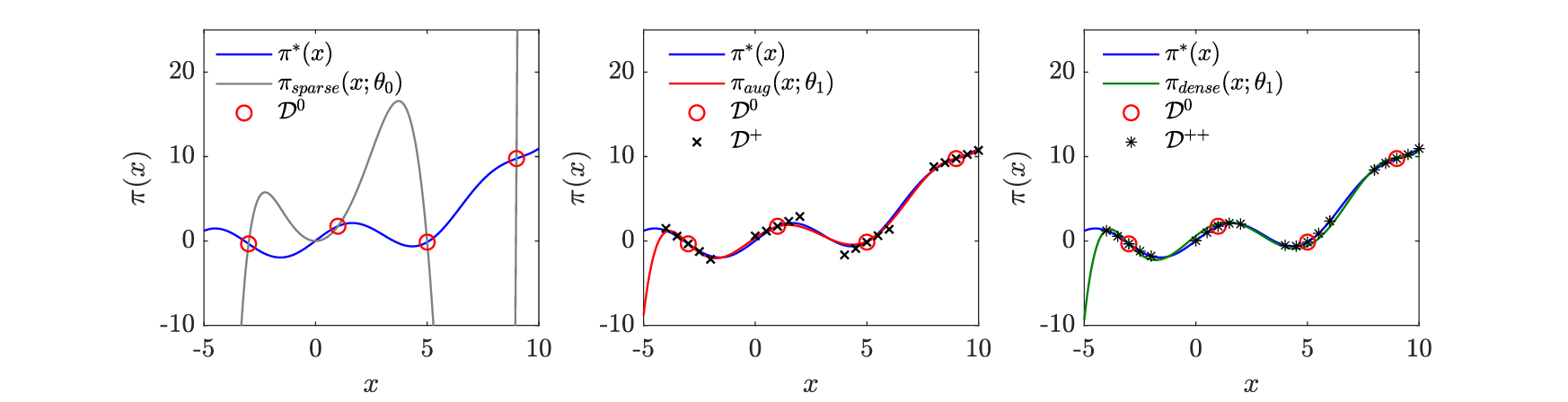}
	\caption{Illustrative example: Comparison of the true policy function (blue) with a 5th order polynomial approximation trained using  sparsely sampled data points (left subplot), proposed augmented data set (middle subplot), and densely sampled data set (right subplot).  }\label{Fig:pathologicalEx}
\end{figure*}
Now, consider the case where around each of these four data points, four more additional  data points are sampled along the tangent (shown in black x in middle subplot). By augmenting the data set $ \mathcal{D}^0 $ with the inexact samples, the fitted polynomial $ \pi_{}(x,\theta_1)$ (shown in red in Fig.~\ref{Fig:pathologicalEx} middle subplot) is now able to approximate the true manifold closely, since the additional data points along the tangent captures the local curvature of the solution manifold. 
This is also compared with the polynomial $ \pi _{}(x,\theta_2) $ fitted using the exact data points sampled by querying the expert at the same points  as the augmented data points (shown in black * and green line in Fig.~\ref{Fig:pathologicalEx} right subplot). This clearly demonstrates that augmenting the data set with additional samples can be beneficial, despite the fact that the augmented  data points may not be exact. %As illustrated in this example, this is because it helps capture the local slope of the solution manifold around each exact sample inferred from the expert query. %we can significantly improve the learning by augmenting the training data set, albeit with inexact data samples. 
Of course, the error due to the inexactness increases as the augmented samples are  further away from the original sample $ x_i $, which will also be studied in this paper.%(that scales by ‖Δx_i ‖^2 as shown in Lemma 1). 

\section{PROPOSED METHOD}\label{sec:Proposed}
As mentioned in the previous section, generating the training data requires solving $ N_{s} $ numerical optimization problems offline, which can be time consuming and computationally expensive. This section leverages the NLP sensitivity to cheaply generate training data samples that can be used to approximate the MPC policy. To keep the notation light, we rewrite the MPC problem \eqref{Eq:MPC} into a standard parametric NLP problem  of the form,
\begin{subequations}\label{Eq:NLP}
\begin{align}
V_{N}(p) =&\min_{\mathbf{w}} \; \;  J(\mathbf{w},p) \label{Eq:NLP:cost}\\
\textup{s.t.}  \;\; & {c}(\mathbf{w},p) = 0 \label{Eq:NLP:system}\\
& g(\mathbf{w},p)\leq 0 \label{Eq:NLP:ineq}
\end{align}
\end{subequations}
where $ p = x(0|t) = x(t) $ is the initial state, the decision variables $ \mathbf{w} := [u(0|t),\dots,u(N-1|t),x(1|t),\dots,x(N|t)]^{\mathsf{T}}  $ the cost \eqref{Eq:MPC:cost} is denoted by \eqref{Eq:NLP:cost}, the system equations \eqref{Eq:MPC:system} are denoted by \eqref{Eq:NLP:system}, the state and input constraints \eqref{Eq:MPC:path}, and the terminal constraint \eqref{Eq:MPC:terminal} are collectively denoted by \eqref{Eq:NLP:ineq}. Since the focus is on solving the MPC problem offline, we drop the time dependency of the initial state, and simply denote the initial condition as $ x $ instead of $ x(t) $. 

The Lagrangian of \eqref{Eq:NLP} is given by
\begin{equation*}
\mathcal{L}(\mathbf{w},p,\lambda,\mu):= J(\mathbf{w},p)  + \lambda^{\mathsf{T}} {c}(\mathbf{w},p) + \mu^{\mathsf{T}}g(\mathbf{w},p)
\end{equation*}
where $ \lambda $ and $ \mu $ are the Lagrangian multipliers of \eqref{Eq:NLP:system} and \eqref{Eq:NLP:ineq} respectively, and the KKT condition for this problem is given by,
\begin{subequations}\label{Eq:KKT}
\begin{align}
\nabla_{\mathbf{w}}\mathcal{L}(\mathbf{w},p,\lambda,\mu) &=0\\
{c}(\mathbf{w},p)  & = 0\\
g(\mathbf{w},p)&\leq 0\\
\mu_{j}g_{j}(\mathbf{w},p)  = 0, \quad \mu_{j} &\geq 0 \quad\forall j
\end{align}
\end{subequations}
Any point $\mathbf{s}^*(p):= [\mathbf{w}^{*},\lambda^{*},\mu^{*}]^{\mathsf{T}}$ that satisfies the KKT conditions \eqref{Eq:KKT} for a given initial condition $ p $ is known  as a KKT point for $ p $. We define the set of  active inequality constraints $ g_{\mathbb{A}}(\mathbf{w},p) \subseteq g(\mathbf{w},p) $ such that $  g_{\mathbb{A}}(\mathbf{w},p)  = 0 $, and  strict complementarity is said to hold if the corresponding Lagrange multipliers $ \mu_{\mathbb{A}} >0 $ (i.e. no weakly active constraint). This set of  KKT conditions can be represented compactly as  \[ \varphi(\mathbf{s}(p),p)  = \begin{bmatrix}
	\nabla_{\mathbf{w}}\mathcal{L}(\mathbf{s},p) \\
	{c}(\mathbf{w},p)  \\
	g_{\mathbb{A}}(\mathbf{w},p)
\end{bmatrix} = 0 \]
\begin{theorem}[\cite{fiacco1976sensitivity}] \label{thm:Sensitivity}
	Let $ J(\cdot,\cdot) $, $ c(\cdot,\cdot) $  and $ g(\cdot,\cdot) $ of the parametric NLP problem $ \mathcal{P}(p) $  be twice continuously differentiable in a neighborhood of the KKT point $ \mathbf{s}^*(p_{0}) $. Further, let linear independence constraint qualification (LICQ), second order sufficient conditions (SOSC) and strict complementarity (SC) hold for the solution vector $ \mathbf{s}^*(p_{0}) $. Then,
	\begin{itemize}
		\item $ \mathbf{s}^*(p_0) $ is a unique local minimizer of $ \mathcal{P}(p_0) $.
		\item For parametric perturbations $ \Delta p $ in the neighborhood of $ p_0 $, there exists a unique, continuous, and differentiable vector function $ \mathbf{s}^*(p_0+\Delta p) $ which is a KKT point satisfying LICQ and SOSC for $ \mathcal{P}(p_0 + \Delta p) $.
		\item There {exist} positive Lipschitz constants $ L_{s} $  and $ L_{V} $ such that the solution vector and the optimal cost \textcolor[rgb]{0,0,1}{satisfy}
		\begin{align}
			\|\mathbf{s}^*(p_0+\Delta p) - \mathbf{s}^*(p_0)\| & \leq L_{s} \|\Delta p\| \label{Eq:Lipschitz} \\
			\|V_{N}(p_0+\Delta p) - V_{N}(p_0)\| &\leq L_{V} \|\Delta p\|   
		\end{align} 
	\end{itemize}
\end{theorem}
\begin{proof}
	See \cite{fiacco1976sensitivity}
\end{proof}

\begin{algorithm}[t]
	\caption{Generating training samples using sensitivity-based data augmentation and approximating the MPC policy.}
	\label{alg:SensitivityTraining}
	\begin{algorithmic}[1]
		\Require $ \mathcal{P}(x) $, $ \mathcal{X}_{feas} $, $ \mathcal{D} = \emptyset$
		\algrulehor
		\For {$ i = 1,\dots,N_{s} $}
		\State Sample $ x_{i} \in \mathcal{X}_{feas} $
		\State $ \mathbf{s}^*(x_{i}) \leftarrow $ Solve $ \mathcal{P}(x_{i}) $
		\State Extract $ u_{i}^* $ from the primal-dual solution vector $  \mathbf{s}^*(x_{i}) $
		\State $\mathcal{D} \leftarrow \mathcal{D} \cup \{(x_{i},u_{i}^*)\}$
		\For {$ j = 1,\dots,N_{p} $}
		\State Sample $ \Delta x_{j} \in \Delta\mathcal{X}_{i} $ in the neighborhood of $ x_{i}$% : p_0 + \Delta{x}_{j} \in \mathcal{X}$
		\State  $  \hat{\mathbf{s}}^*(x_{i}+\Delta x_{j}) =  \mathbf{s}^*(x_{i})  - \mathcal{M}^{-1}\mathcal{N}\Delta{x}_{j}$
		\State Extract $ \hat{u}_{j}^* $ from the solution vector $  \mathbf{\hat s}^*(x_{i}+ \Delta{x}_{j}) $
		\If {$ {g}_{\mathbb{A}}({\mathbf{w}^*(x_{i})}) == {g}_{\mathbb{A}}({\hat{\mathbf{w}}(x_{i}+ \Delta x_{j})}) $}
		\State $\mathcal{D} \leftarrow \mathcal{D} \cup \{(x_{i}+\Delta x_{j},\hat u_{j}^*)\}$
		\EndIf
		\EndFor
		\EndFor
		\State $ {\theta}_{2} \leftarrow\arg \min_{\theta} \frac{1}{N_{s}N_{p}}\sum_{i=1}^{N_{s}N_{p}} \| \pi_{\textup{approx}}(x_{i};\theta) - u^*_{i} \|^2  $
		\algrulehor
		\Ensure $ \pi_{\textup{approx}}(x;{\theta}_{2}) $
	\end{algorithmic}
\end{algorithm}

Since the implicit function $ \mathbf{s}^*(p) $  satisfies  $ \varphi(\mathbf{s}^*(p),p) = 0$ for any $ p $ in the neighborhood of $ p_{0} $,  the implicit function theorem gives 
\begin{align}\label{Eq:IFT}
	\frac{\partial}{\partial p} \varphi(\mathbf{s}^*(p),p) \bigg|_{p=p_{0}}  = \frac{\partial  \varphi}{\partial \mathbf{s}}  \frac{\partial\mathbf{s}^*}{\partial p} + \frac{\partial \varphi}{\partial p} = 
\mathcal{M} 	\frac{\partial\mathbf{s}^*}{\partial p} + \mathcal{N}= 0
\end{align}
where \[  \mathcal{M}:= \begin{bmatrix}
	\nabla^2_{\mathbf{w}\mathbf{w}}\mathcal{L}({\mathbf{s}^*(p_0)}) &\nabla_{\mathbf{w}}{c}({\mathbf{w}^*(p_0)}) & \nabla_{\mathbf{w}}{g}_{\mathbb{A}}({\mathbf{w}^*(p_0)})\\ 
	\nabla_{\mathbf{w}}{c}({\mathbf{w}^*({p_0})})^{\mathsf{T}} & 0  & 0 \\ 
	\nabla_{\mathbf{w}}{g}_{\mathbb{A}}({\mathbf{w}^*(p_0)})^{\mathsf{T}}& 0 &0%{g}({{w}^*({p}_{0})})
	\end{bmatrix} \] is the KKT matrix, and
	\[ \mathcal{N} :=  \begin{bmatrix}
\nabla^2_{\mathbf{w}{p}}\mathcal{L}({\mathbf{s}^*(p_0)})\\ 
\nabla_{{p}}{c}({\mathbf{w}^*(p_0)})^{\mathsf{T}}  \\ 
\nabla_{{p}}{g}_{\mathbb{A}}({\mathbf{w}^*(p_0)})^{\mathsf{T}}
\end{bmatrix} \]
Linearizing the solution vector $ \mathbf{s}^*(p) $ around $ p_0 $ gives
\begin{align*}
\mathbf{s}^*(p_0+\Delta p) =\mathbf{s}^*(p_0) + \frac{\partial \mathbf{s}^*}{\partial p} \Delta p + \mathcal{O}(\|\Delta p\|^2)
\end{align*}
Ignoring the higher order terms, the solution of the neighboring problems $ p_0+\Delta p $ can be approximated using \eqref{Eq:IFT} as,
\begin{equation}\label{Eq:SensitivityUpdate1}
	\hat{\mathbf{s}}^*(p_{0}+\Delta p)   = \mathbf{s}^*(p_0)  -\mathcal{M}^{-1}\mathcal{N}\Delta p
\end{equation}
where $ \hat{\mathbf{s}}^*(p_{0}+\Delta p)  $ is the approximate primal-dual solution of the optimization problem  $ \mathcal{P}(p_{0}+ \Delta p) $, and $ \Delta \mathbf{s}^* := -\mathcal{M}^{-1}\mathcal{N}\Delta p $ is 
known as the \textit{tangential predictor} or \textit{linear predictor}. Simply put, we linearize the solution manifold $ \mathbf{s}^*(p) $ around $ p_{0} $, and compute the approximate solution at $ p_{0}+ \Delta p $. Note that since SOSC holds, the KKT matrix $ \mathcal{M} $ is invertible. 

Computing $ \Delta \mathbf{s}^* := -\mathcal{M}^{-1}\mathcal{N}\Delta p $ requires only a linear solve, which is significantly cheaper to compute than solving the full NLP.
%\subsection{Sensitivity-based data augmentation}
From this we can see that once the solution to the NLP problem $ \mathcal{P}(x_{i}) $ is available for a given initial state $ x_{i} \in \mathcal{X}_{feas}$, we can exploit the parametric property of the NLP to compute a fast approximate solution for an additional finite set of $ j = 1,\dots,N_{p} $  optimization problems  $ \mathcal{P}(x_{i}+\Delta{x}_{j}) $ with initial states $ x_{i} + \Delta x_{j} \in \mathcal
X_{feas}$ in the neighborhood of $ x_{i} $. More precisely, $ \Delta x_{j} $ is sampled from a subset of arbitrary size $ \Delta \mathcal{X}_{i} \subset \mathcal{X}_{feas}$ such that $ x_{i} \in \text{int}(\Delta\mathcal{X}_{i}) $. Using the tangential predictor \eqref{Eq:SensitivityUpdate1}, the corresponding optimal solution denoted by  $ \hat{u}^*_{j} $ can then be  evaluated, which only requires a linear solve.
By exploiting the sensitivities, one can then generate $ M:=N_{s}N_{p} $ number of training samples  using only  $ N_{s} $ NLP problems solved exactly.  
The pseudo-code for the proposed sensitivity-based data augmentation technique for MPC policy approximation is summarized in Algorithm~\ref{alg:SensitivityTraining}. The idea of exploiting the parametric property of the MPC problem with respect to the initial states  $ p = x(t) $ is also used  in other parts of MPC literature such as  real-time iteration \cite{diehl2005RTI}, advanced-step MPC \cite{zavala2009advanced,jaschke2014fast} and adaptive horizon MPC \cite{griffith2018robustly,DK2020CMU} to name a few.

\begin{remark}[Change in active constraint set]
	Changes in the active set induces non-smooth points in the solution manifold, which cannot be captured by the tangential predictor. 
	If the perturbation $ \Delta x_{j} $ induces a change in the set of active constraints, then  one would have to solve a quadratic programming problem, often known as predictor QP,  in order to produce a piecewise linear prediction manifold that can capture the non-smooth \enquote{corners} in the solution manifold \cite{bonnans1998optimization}. Depending on the problem size and complexity, this may still be computationally cheaper than solving a full NLP problem. A simpler alternative  is to discard any sensitivity updates that induce a change in active constraint set. By doing so, we do not augment the data set $ \mathcal{D} $ with points that induce a change in the active set. This  is what is adopted in this paper (cf. line 10 in Algorithm~\ref{alg:SensitivityTraining}). 
\end{remark}

A natural question that  then arises is, how does augmenting the data set with inexact samples affect the policy approximation. To study this, consider an optimal solution manifold of an NLP denoted by $ \mathbf{s}^*(p) $ that we wish to approximate. Assume $ p  \in  \mathcal{X}_{feas}$ is sampled at $ N_{s} $ discrete points, and the corresponding optimal solution $ \mathbf{s}^*(p_{i}) $ for all $ i= 1,\dots, N_{s} $ is obtained by solving the NLP exactly. Now consider a piecewise linear inexact solution manifold  $ \hat{\mathbf{s}}(p) $ that is given by the tangential predictor \eqref{Eq:SensitivityUpdate1} using  the exact solution at $ p_{i} $ in the $ i^{th} $ linear region for all $ i = 1,\dots,N_{s} $. Let $ \Delta p_{i} $ be  defined around each sample point $ p_{i} $, such that $ \hat{\mathbf{s}}(p) $ exists for all $ p\in \mathcal{X}_{feas} $. The true solution manifold $ \mathbf{s}^*(p) $, and the piecewise linear inexact solution manifold $\hat{ \mathbf{s}}(p) $ generated around the $ N_{s} $ samples is graphically illustrated for a one-dimensional case in Fig.~\ref{Fig:manifold}. The maximum deviation between the true solution manifold $ \mathbf{s}^*(p) $ and the inexact manifold $ \hat{ \mathbf{s}}(p) $ is quantified in the following Lemma. 
\begin{figure}
	\centering
	\includegraphics[width=0.67\linewidth]{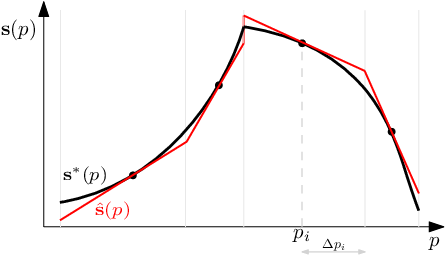}
	\caption{Graphical illustration of a one dimensional solution manifold $ \mathbf{s}^*(p)$ (black), and a piecewise linear inexact manifold $ \hat{\mathbf{s}}(p) $ (red) approximated around  the solutions at $ p_{i} $ (black dot) using a tangential predictor (shown here for  $ N_{s}  = 4 $). }\label{Fig:manifold}
\end{figure}
\begin{lemma}\label{thm:MaxD}
	Given a solution manifold $ \mathbf{s}^*(p) $, and a piecewise linear inexact manifold $ \hat{\mathbf{s}}(p) $  approximated around  $ \mathbf{s}^*(p_{i}) $ using the tangential predictor \eqref{Eq:SensitivityUpdate1} for all $ i = 1,\dots,N_{s} $,  with $  \|\Delta p_{i}\|  $ chosen around each $ p_{i} $ such that  the active constraint set remains the same and $ \exists\, \hat{\mathbf{s}}(p) \;\; \forall p \in \mathcal{X}_{feas}$, then the following holds
	\begin{equation}\label{Eq:ManifoldErrrorBound}
		\|\hat{\mathbf{s}}(p) - \mathbf{s}^*(p)\| \leq \max\left(  \left\{ L_{p_{i}} \|\Delta p_{i}\|^2\right\}_{i=1}^{N_{s}}\right) 
	\end{equation}
for some positive constants $ L_{p_{i}} $
\end{lemma}
\begin{proof}
  From the continuity and  differentiability of $ \mathbf{s}^*(p) $ around each sample point $ \mathbf{s}^*(p_{i}) $ (cf. Theorem~\ref{thm:Sensitivity}), there exists some positive Lipschitz constant $ L_{p_{i}} $ for the $ i^{th} $ piecewise linear region such that,
	\[ 	\|\hat{\mathbf{s}}(p_{i}) - \mathbf{s}^*(p_{i})\| \leq L_{p_{i}} \|\Delta p_{i}\|^2\] 
	for all $ i = 1,\dots,N_{s} $. 
	Aggregating over all the regions, the max-distance between the  true manifold and the inexact  manifold at any point $ p $ is then given by 
	\[ \|\hat{\mathbf{s}}(p) - \mathbf{s}^*(p)\| \leq \max\left(  \left\{ L_{p_{i}} \|\Delta p_{i}\|^2\right\}_{i=1}^{N_{s}}\right)  \]
\end{proof}

We now want to approximate the solution manifold with any suitable parametric function $ \pi(p;\theta) $ using $ M := N_{s}N_{p}$ data samples.
\begin{defn}[Sufficiently rich parametrization]\label{def:RichParametrization}
Given a parametric function $ \pi(x;\theta) $ that is used to approximate a function $ F(x) $ over  some domain $ x \in \mathbb{X} $, the parametric function $ \pi(x;\theta) $ is said to be sufficiently richly parameterized if there exists $ \theta $ in some searchable domain $ \Theta $ such that $ F(x) = \pi(x;\theta) $ for all $ x \in \mathbb{X} $.
\end{defn}
\begin{assume}\label{asm:RichParametrization}
	The functional form of $ \pi(p;\theta) $ has sufficiently rich parametrization and $ \exists \; \theta^* $ such that $ \mathbf{s}^*(p) = \pi(p;\theta^*) $, and $ \exists \; \hat{\theta} $ such that $ \hat{\mathbf{s}}(p) = \pi(p;\hat{\theta}) $.
\end{assume}
 Consider the case where the training data is sampled by solving the NLP exactly at the $ M $ samples, which gives us \begin{equation}\label{Eq:Estimator1}
 	\theta_1 = \arg \min_\theta \sum_{i=1}^{M} \|\pi(p_{i};\theta) - \mathbf{s}^*(p_{i})\|^2 
 \end{equation}
 Now consider the case where the training data is sampled from the inexact manifold $ \hat{\mathbf{s}}(p) $ at the same $ M $ samples, which gives us
 \begin{equation}\label{Eq:Estimator2}
 	\theta_2= \arg \min_\theta \sum_{i=1}^{M} \|\pi(p_{i};\theta) - \hat{\mathbf{s}}(p_{i})\|^2 
 \end{equation}
 The following result then establishes the error bound between the function approximators $ \pi(p;\theta_1) $ and $ \pi(p;\theta_2) $ stemming solely from augmenting the data set with inexact samples. 
\begin{theorem}\label{thm:ErrorBounds}
	Given Assumption~\ref{asm:RichParametrization} for the problem setup as in Lemma~1, if $ \theta_1 $ and $ \theta_2 $ are consistent estimators of \eqref{Eq:Estimator1} and \eqref{Eq:Estimator2} respectively, then 
	\begin{equation}\label{Eq:thm:bounds}
		\|\pi(p;\theta_1) - \pi(p;\theta_2)\| \leq D 
	\end{equation}
in probability as $ M \rightarrow \infty $, where $ D:= \max\left(  \left\{ L_{p_{i}} \|\Delta p_{i}\|^2\right\}_{i=1}^{N_{s}}\right)  $
\end{theorem}
\begin{proof}
	From Assumption~\ref{asm:RichParametrization} and Lemma~\ref{thm:MaxD}, we have that
	\begin{align*}
	\|\hat{\mathbf{s}}(p) - \mathbf{s}^*(p)\| =& \|\pi(p;\theta^*) - \pi(p,\hat{\theta})\|\leq D
	\end{align*}
If $ \theta_1 $ is a consistent estimator of  \eqref{Eq:Estimator1}, then 
\[ \lim_{M\rightarrow\infty}  \mathbb{P}(|\theta_1 - \theta^*|>\epsilon) = 0\]
for any $ \epsilon \geq0 $ where $ \mathbb{P}(\cdot) $ denotes the probability. Similarly, we have 
\[ \lim_{M\rightarrow\infty}  \mathbb{P}(|\theta_2 - \hat{\theta}|>\epsilon) = 0\]
if  $ \theta_2 $ is a consistent estimator of  \eqref{Eq:Estimator2}.
Combining these we get the inequality in \eqref{Eq:thm:bounds} in probability as $ M \rightarrow \infty $.
\end{proof}
%Note that we only consider the limiting case with $ M \rightarrow \infty $ here, since we are mainly interested in quantifying the error due to the inexact samples, and not the  approximation error. 
Note that although the above result considers  $ M\rightarrow \infty $, generating the training samples for \eqref{Eq:Estimator2} only requires solving $ N_{s} $ NLP problems exactly, the remaining $ M-N_{s} $ training samples are given by the tangential predictor \eqref{Eq:SensitivityUpdate1}. 
Since $ \mathbf{s}^*(p) $ is the true solution manifold, solving the MPC policy online would  also give us $ \mathbf{s}^*(p) $. Consequently, the case of $ M  \rightarrow \infty$ in \eqref{Eq:Estimator1} can  be seen as the MPC policy used online, and using the arguments of Theorem~\ref{thm:ErrorBounds} we can state the following corollary that establishes the error bound between the MPC policy and the approximate policy trained using the inexact samples.
\begin{corollary}\label{thm:corollary}
	Given Assumption~\ref{asm:RichParametrization} for the problem setup as in Lemma~\ref{thm:MaxD}, if $ \theta_2 $ is a consistent estimator of \eqref{Eq:Estimator2}, then 
	\begin{equation}\label{Eq:thm:bounds2}
		\|\pi_{\textup{mpc}}(p) - \pi(p;\theta_2)\| \leq D 
	\end{equation}
in probability as $ M \rightarrow \infty $.
\end{corollary}
Theorem~\ref{thm:ErrorBounds}/Corollary~\ref{thm:corollary} thus provides an upper bound on the error solely induced due to augmenting the data set with a  very large number of inexact samples  based on  $ N_{s} $ exact samples.  The following (obvious) result  considers the effect of the number of samples $ N_{s} $ around which the piecewise linear  manifold is generated.
\begin{theorem}\label{thm:NsInfty}
	Given the problem setup as in Lemma~1, and the estimators \eqref{Eq:Estimator1} and \eqref{Eq:Estimator2},
	\begin{equation}\label{Eq:Corollary}
		\lim_{{N_{s}} \rightarrow \infty} \|\pi(p;\theta_1) - \pi(p;\theta_2)\|  = 0
	\end{equation}
\end{theorem}
\begin{proof}
	The proof of this result follows trivially from Lemma~{1} where, as $ N_{s} \rightarrow \infty $, the neighborhood of the tangential predictor gets smaller, i.e. $ \Delta p_{i} \rightarrow 0 \quad \forall i$, and we have
	\[  \lim_{{N_{s}} \rightarrow \infty} \|\hat{\mathbf{s}}(p) - \mathbf{s}^*(p)\|  = 0 \]
	This implies that  \eqref{Eq:Estimator1} and \eqref{Eq:Estimator2} are identical as $ {N_{s}} \rightarrow \infty $, which proves our result.
\end{proof}

To summarize, a large size of $ \Delta p_{i}  $  implies that one can cheaply obtain  several inexact data samples covering a larger subset of the feasible state-space. However, as the size of $ \Delta p_{i}  $ grows,  the approximation error induced by the sensitivity update also increases, as quantified by Theorem~\ref{thm:ErrorBounds}. %This confirms and quantifies that the error due to augmenting the data set with inexact samples depends on the maximum size of the neighborhood $ \Delta p_{i} $ around an NLP sample. 
 As such Theorem~\ref{thm:ErrorBounds} and Theorem~\ref{thm:NsInfty} establish the trade-off between accuracy and  computational cost of generating the training samples. 
 
 \begin{remark}[Sampling]
 	The proposed sensitivity-based data augmentation is not dependent on any particular sampling strategy, and can be used with different sampling strategies as one would have used with approximating the MPC policy without data augmentation. For any given sampling strategy, instead of solving all the sampled states exactly, one can instead use the tangential predictor \eqref{Eq:SensitivityUpdate1} at the samples that are within a user-defined neighborhood of an already existing sample. 
 	
 	\begin{remark}[Setpoint and weights]
 		The proposed sensitivity-based data augmentation is not restricted to parametric NLPs w.r.t. the initial states $ p_{i} = x _{i}$, but can also be utilized  by parameterizing the optimization problem with respect to other parameters such as  reference trajectories $ x^{sp}_{i} $,  MPC tuning parameters  such as weights in the cost function $ \omega_{i} $, or measured disturbances $ d_{i} $ in addition to the initial states $ x_{i}$, i.e. $ p_{i} = [x_{i},x^{sp}_{i},\omega_i,d_{i}]^{\mathsf{T}} $. 
 	\end{remark}

 \begin{remark} [Linear MPC]
 	If $ \ell(\cdot,\cdot) $ is convex quadratic and $ f(\cdot) $ is linear,  then $ \hat{\mathbf{s}}^*(x+\Delta x)   = {\mathbf{s}}^*(x+\Delta x)   $, and consequently $ D = 0 $ in Theorem~\ref{thm:ErrorBounds}. 
 \end{remark}
 	
 	% sobol sampling - stratfied sampling where the sttate space is divided into active sets,
 	
 	% non probability sampling methods - such as convinece sampling - based on availability - for example expert demonstration
 	
 	% control oriented sampling - However there is an exponential decay, so cant be used much... Ns*N + Ns*Np
 \end{remark}

The proposed approach is also not restricted to the MPC formulation \eqref{Eq:MPC}, but can also be used with other variants of MPC formulation that typically involves solving nonlinear programming (NLP) problems, such as robust MPC \cite{paulson2020approximate,hertneck2018learning}, and multistage scenario-based MPC \cite{lucia2018deep} with moving horizon estimation \cite{karg2021approximateMHE} etc. to name a few. 
{It can also be seen from \eqref{Eq:SensitivityUpdate1} that the tangential predictor provides the primal-dual solution. Therefore, the proposed data augmentation scheme can also be used to augment a training data set consisting of the optimal dual variables. This is useful in cases where one would like to learn the dual policy such as  in \cite{zhang2019safe}. 
}

\section{ILLUSTRATIVE EXAMPLES}\label{sec:Example}
\subsection{Benchmark CSTR}
\begin{figure}
	\centering
	\includegraphics[width=0.99\linewidth]{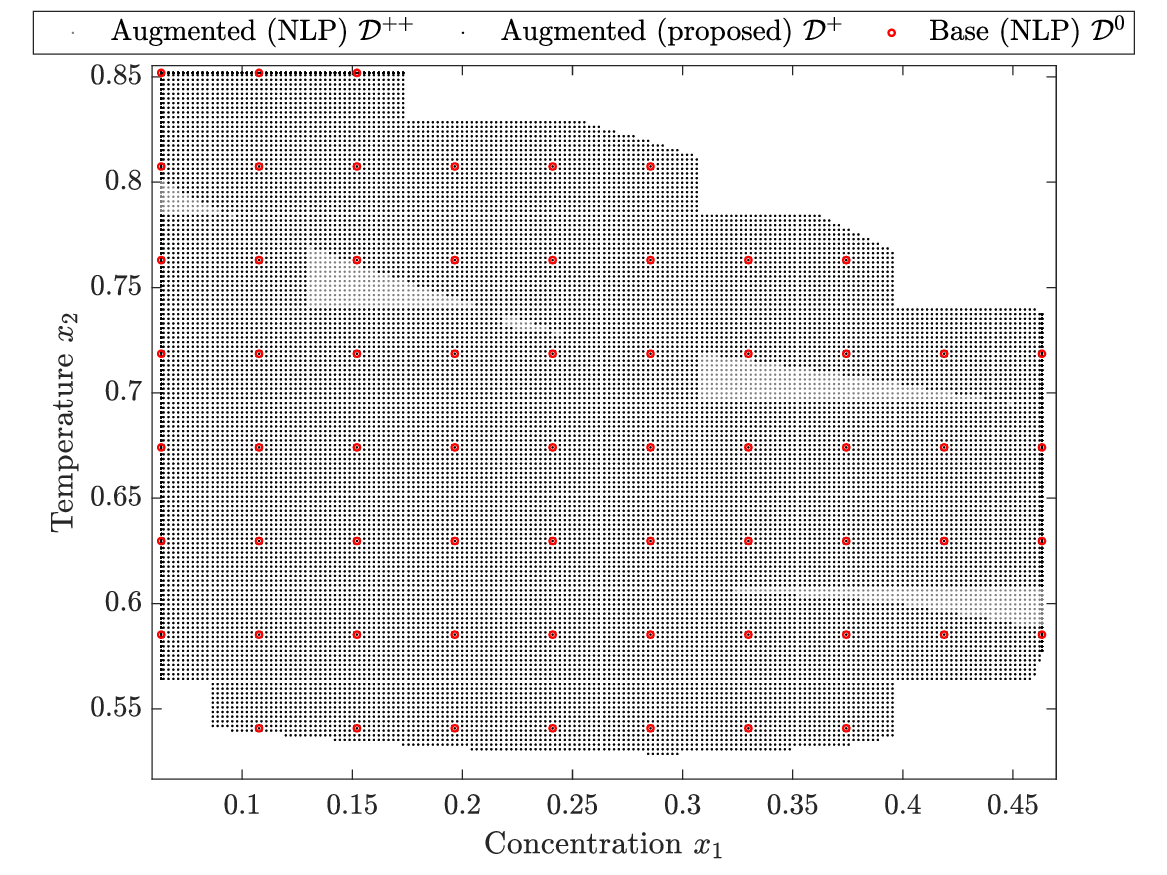}
	\caption{Example A: Grid-based sampling of the feasible state space $ \mathcal{X}_{feas} $. Red circles denote the base samples $ \mathcal{D}^0 $, where the corresponding optimal input is generated by solving the full optimization problem,  the black dots denotes the augmented samples $ \mathcal{D}^+ $ where the corresponding optimal input  is generated using the tangential predictor, and the gray dots denote the samples   in $ \mathcal{D}^{++} $ where the corresponding optimal input  is generated by solving the full NLP. }\label{Fig:Grid}
\end{figure}

We now apply the proposed approach on a benchmark CSTR problem from \cite{mayne2011tube} that was also used in the context of  MPC policy approximation in \cite{hertneck2018learning}. 
This problem consists of two states, namely the {scaled} concentration and reactor temperature (denoted by $ x_{1} $ and $ x_{2} $, respectively). The process is controlled using the coolant flow rate $ u $.  The model  is given by
\begin{align*}
\dot{x}_{1} &= (1/\tau)(1-x_{1}) - k x_{1} e^{-\beta/x_{2}}\\
\dot{x}_{2} &= (1/\tau)(x_{f}-x_{2}) + k x_{1} e^{-\beta/x_{2}} - \alpha u (x_{2} -x_{c})
\end{align*}
and the model parameters are $ \tau = 20 $, $ k=300$, $ \beta=5 $, $ x_{f}=0.3947 $, $ x_{c} = 0.3816 $, and $ \alpha = 0.117 $. Furthermore, we have $ \mathcal{X} = [0.0632,0.4632]\times [0.4519,0.8519] $ and $ \mathcal{U} = [0,2] $. The setpoint is given by $ x^{sp} = [0.2632, 0.6519]^{\mathsf{T}} $.
The stage cost is given by \[  \ell(x,u) = \|x-x^{sp}\|^2 + 10^{-4} \|u\|^2  \] The MPC problem is solved with a sampling time of 3s and a  prediction horizon of $ N=140 $.

One approach to generate the learning samples is to use a grid-based sampling approach as done in \cite{hertneck2018learning}, where the optimal input $ u^* = \pi_{\textup{mpc}}(x) $ is  evaluated at each grid point. In general, a small grid size is preferred since this would improve the MPC policy approximation. However, this would lead to large sample size $ N_{s} $. 
The proposed approach enables us to choose a relatively  larger grid size, where the corresponding optimal input $ \pi_{\textup{mpc}}(x) $ is evaluated by solving the optimization problem. Additional grid points can then be generated with a smaller grid size around each grid point, and the corresponding optimal input can be computed by using the tangential predictor \eqref{Eq:SensitivityUpdate1}. 

In this case, we first generated $ N_{s} = 64$ samples using a sparse grid with an interval of 0.0445 for both $ x_{1} $ and $ x_{2} $. This base data set is denoted by $ \mathcal{D}^0 $, which is shown in red circles in Fig.~\ref{Fig:Grid}. Since this is too few samples to approximate the MPC policy, we then augment the data set with additional samples, where around each grid point, the state space is further sampled with a smaller grid size  of 0.0052 for both $ x_{1} $ and  $ x_{2} $.  For these additional grid points, we obtained the corresponding inexact optimal input by using the tangential predictor \eqref{Eq:SensitivityUpdate1} which is shown in black dots in Fig.~\ref{Fig:Grid}. The additional grid points that induced a change in the active set were simply discarded. By doing so, we were able to cheaply generate and augment 23516 additional data points using only 64 full NLP computations. This data  set is denoted by $ \mathcal{D}^+ $. 
To serve as benchmark, the optimal input at the  additional grid points were also generated  using Algorithm~\ref{alg:Training}, i.e. by solving the full NLP (shown in gray dots in Fig.~\ref{Fig:Grid}).
Since the changes in the active constraints are not an issue when solving the full NLP, this approach generated an additional 24503 data points, which comes at a very high cost of computation. This data set is denoted by $ \mathcal{D}^{++} $. The total number of data points, as well as the average and cumulative CPU time for generating the training data sets $ \mathcal{D}^0 $, $ \mathcal{D}^+ $ and $ \mathcal{D}^{++} $ are summarized in Table~\ref{tb:CPUtime}. This clearly shows the benefit of the proposed approach in terms of the computational cost of generating the training samples.

\begin{table}[t]
	\begin{footnotesize}
		\begin{center}
			\caption{Example A: CPU time in [s] for generating the training samples.}\label{tb:CPUtime}
			\begin{tabular}{l|cc|c}
				\toprule
				& Cumulative  & Average   &total \\ 
				& CPU time & CPU time per  & no. of\\ 
				& [s]& data point [s]&samples \\ \hline
				Base grid (NLP) $ \mathcal{D}^0 $ &  35.766  & 0.5589  &  64\\
				Augmented   (proposed)  $ \mathcal{D}^+ $& 175.77 & 0.0074  & 23516 \\
				Augmented  (NLP) $ \mathcal{D}^{++} $& 12960.4& 0.529 & 24503 \\
				\bottomrule
			\end{tabular}
		\end{center}
	\end{footnotesize}
\end{table}

\begin{figure}
	\centering
	\begin{subfigure}{\linewidth}
		\includegraphics[width=\linewidth]{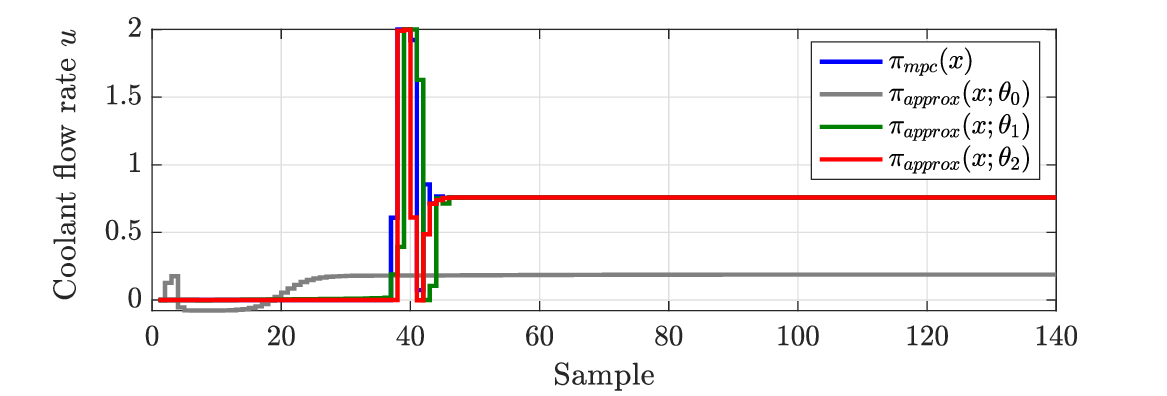}
		\caption{}\label{Fig:CSTRsim}
	\end{subfigure}
	\begin{subfigure}{\linewidth}
		\includegraphics[width=\linewidth]{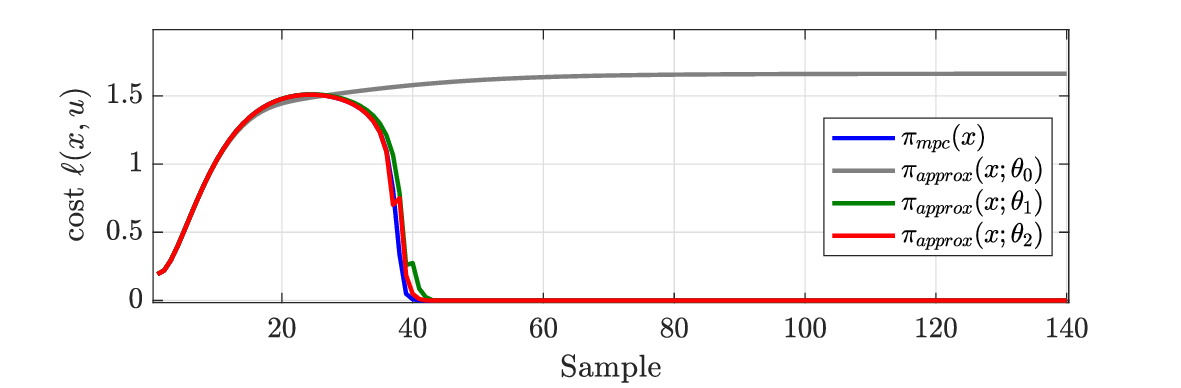}
		\caption{}\label{Fig:CSTRcost}
	\end{subfigure}
	\caption{Example A: Closed-loop simulation results comparing the {(a) control actions and (b) cost $ \ell(x,u) $} given by the 
		MPC  policy $ \pi_{\textup{mpc}}(x) $ (blue) and the approximate  policies $ \pi_{\textup{approx}}(x;\theta_0) $  trained using sparsely sampled data (gray), $ \pi_{\textup{approx}}(x;\theta_1) $  trained using finely sampled data obtained by solving full NLP (green),  and $ \pi_{\textup{approx}}(x;\theta_2) $ trained using inexact samples using the proposed data augmentation approach (red).}
\end{figure}

Using the generated data sets, we approximate the MPC policy using deep neural networks with $  5 $ hidden layers and $10 $ neurons in each hidden layer with the hyperbolic tangent sigmoid  as the activation function in each neuron. Note that the type of the function approximator, its hyperparameters, and the specifics of the MSE minimization problem are not the focus of this paper, and one may find an alternative/better network architecture than the one used here, for example by using Bayesian optimization \cite{snoek2012practical}. The NLP problems  $ \mathcal{P}(x) $ were solved using \texttt{IPOPT} \cite{wachter2006ipopt} with \texttt{MUMPS} linear solver.  The neural network was trained using  \texttt{fitnet} in \texttt{MATLAB v2020b}.  All computations were performed on a 2.6 GHz processor with 16GB memory. Source codes for the simulation results presented in this technical note can be found in the GitHub repository \texttt{https://github.com/dinesh-krishnamoorthy/Sensit
	ivity-DataAugmentation}.  

First we approximate the MPC policy using only the sparsely sampled data set $ \mathcal{D}^0 $, which gives the control policy $ \pi_{\textup{approx}}(x;\theta_0) $. We then approximate the MPC policy using the computationally expensive full data set  $ \mathcal{D}_{1} :=  \mathcal{D}^0 \cup \mathcal{D}^{++} $, which gives the approximate policy $ \pi_{\textup{approx}}(x;\theta_1) $.   We then    approximate the MPC policy using the proposed sensitivity-based augmented data set $ \mathcal{D}_{2} :=  \mathcal{D}^0 \cup \mathcal{D}^+ $, which gives the approximate policy $ \pi_{\textup{approx}}(x;\theta_2) $. 

{Fig.~\ref{Fig:CSTRsim} compares the closed-loop control actions provided by the different policies. The control trajectory provided by the MPC policy $ \pi_{\textup{mpc}}(x) $ (shown in blue) serves as the ideal benchmark. The closed-loop control trajectory using the approximate policy  $ \pi_{\textup{approx}}(x;\theta_0) $ trained using the sparsely sampled data set  $ \mathcal{D}^0 $ is shown in gray, where it can be seen that $ \mathcal{D}^0 $ fails to approximate the MPC policy due to insufficient training data samples.  
The closed-loop control trajectory using the approximate policy  $ \pi_{\textup{approx}}(x;\theta_1) $ trained using the full NLP data set $ \mathcal{D}_{1} $ is shown in green, which is able to approximate the MPC policy well as one would expect.  Finally, the closed-loop control trajectory using the approximate policy $ \pi_{\textup{approx}}(x;\theta_2) $ trained using the proposed augmented data set $ \mathcal{D}_{2}$ is shown in red, which is also able to approximate the MPC policy closely, although the approximate policy $ \pi_{\textup{approx}}(x;\theta_2) $ is trained using inexact samples, that are significantly cheaper to generate than the full NLP (cf. Table~\ref{tb:CPUtime}). }

{The corresponding closed-loop cost $ \ell(x,u) $ is  also shown in Fig.~\ref{Fig:CSTRcost}. Here it can be clearly seen that the  performance using the approximate policy $ \pi_{\textup{approx}}(x;\theta_2) $  trained using the proposed data augmentation scheme is almost identical to the performance of the MPC policy, and the approximate policy  $ \pi_{\textup{approx}}(x;\theta_1) $ trained using the full NLP. }

{Next, we compare the closed-loop trajectories  when starting from different initial conditions. The state trajectories using the MPC policy $ \pi_{\textup{mpc}}(x) $~(blue) and the approximate  policies $ \pi_{\textup{approx}}(x;\theta_0) $~(gray), $ \pi_{\textup{approx}}(x;\theta_1) $~(green dotted lines), and $ \pi_{\textup{approx}}(x;\theta_2) $ ~(red dashed lines) are shown in Fig.~\ref{Fig:CSTRtraj}. Here it can be clearly seen that the state trajectories  obtained using the proposed data augmentation scheme  (red dashed lines) is almost identical to the MPC policy (blue) and the approximate policy trained using the full NLP (green dotted lines) for all the different initial conditions, whereas the trajectories obtained by using the approximate policy trained using  sparsely sampled data (gray) deviates significantly from the other trajectories. 
}
\begin{figure}
	\centering
	\includegraphics[width=\linewidth]{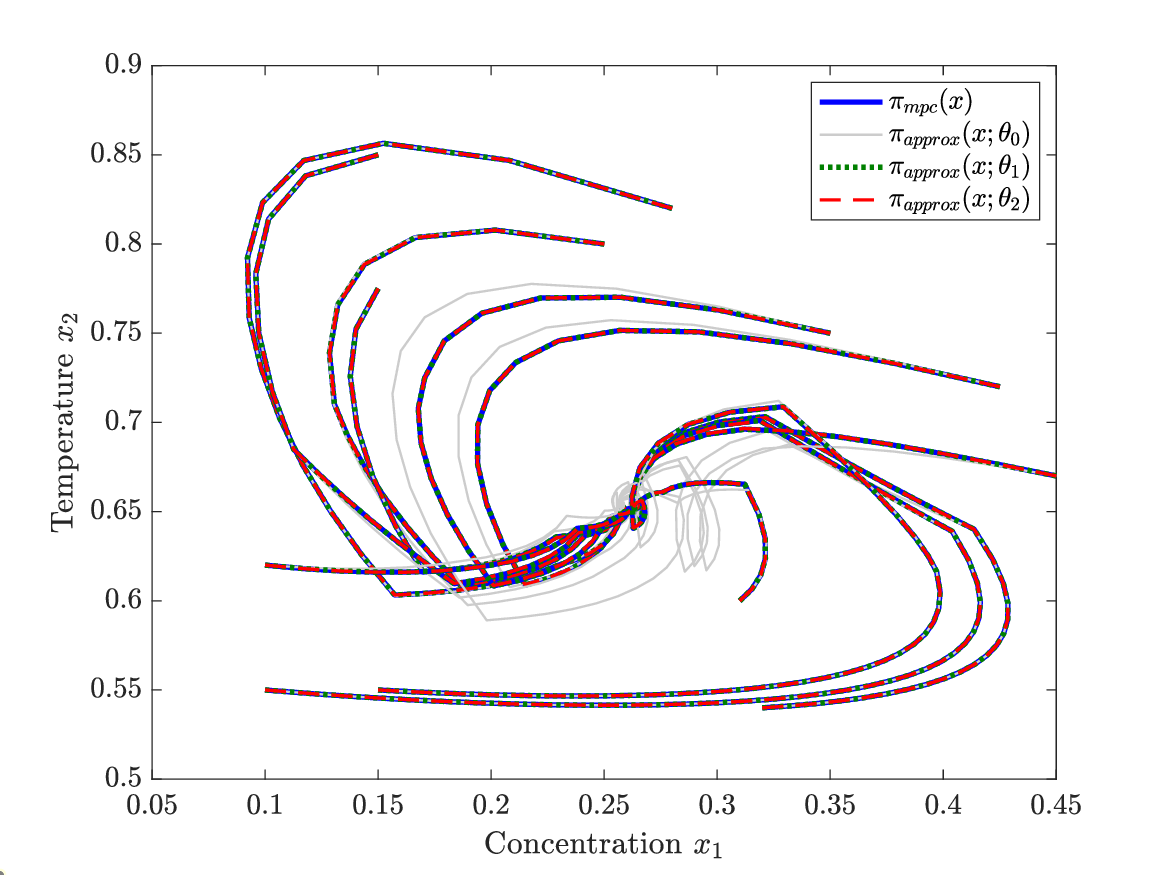}
	\caption{{Example A: State trajectories starting from different initial conditions  given by the 
			MPC  policy $ \pi_{\textup{mpc}}(x) $ (blue) and the approximate  policies $ \pi_{\textup{approx}}(x;\theta_0) $  trained using sparsely sampled data (gray), $ \pi_{\textup{approx}}(x;\theta_1) $  trained using finely sampled data obtained by solving full NLP (green dotted),  and $ \pi_{\textup{approx}}(x;\theta_2) $ trained using inexact samples using the proposed data augmentation approach (red dashed).}}\label{Fig:CSTRtraj}
\end{figure}

\begin{figure*}
	\centering
	\includegraphics[width=\linewidth]{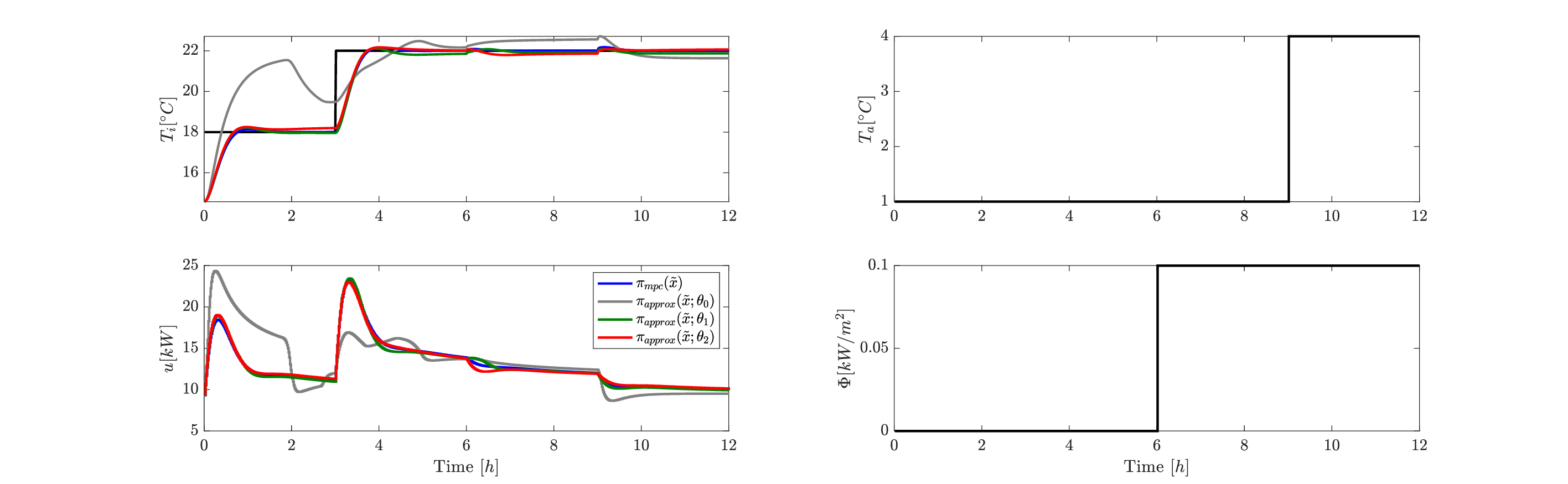}
	\caption{Example B: Closed-loop simulation results comparing performance of the 
		MPC  policy $ \pi_{\textup{mpc}}(\tilde x) $ (blue) and the approximate  policies $ \pi_{\textup{approx}}(\tilde{x};\theta_0) $  trained using sparsely sampled data (gray), $ \pi_{\textup{approx}}(\tilde{x};\theta_1) $  trained using finely sampled data obtained by solving full NLP (green),  and $ \pi_{\textup{approx}}(\tilde{x};\theta_2) $ trained using inexact samples using the proposed data augmentation approach (red).}\label{Fig:BCCsim}
\end{figure*}

This clearly demonstrates that by using the proposed data augmentation approach,  we can sparsely sample $ \mathcal{X}_{feas} $, which  reduces the number of optimization problems that needs to be solved offline, and augment the data set with several additional inexact samples that can be obtained cheaply. Consequently, the overall time and computational cost required to generate the training samples is significantly  lesser, and at the same time, the approximate policy trained using the augmented samples provides similar performance as the MPC policy.

\addtolength{\textheight}{-0cm}

\subsection{Building Climate Control}
We now illustrate the proposed approach on a building climate control problem, for which there have been several works considering MPC as the control strategy, see \cite{drgovna2018approximate,privara2011role} and the references therein. 
In our simulations, we model the heat dynamics of a building based on the modeling framework from  \cite{bacher2011identifying}, as shown below,
%\begin{subequations}\label{Eq:BCC}
		\begin{align*}
\frac{\textup{d} T_{s}}{\textup{d}t} & = \frac{1}{R_{is}C_{s}}(T_{i} - T_{s})\\
\frac{\textup{d} T_{i}}{\textup{d}t} & = \frac{1}{R_{is}C_{i}}(T_{s} - T_{i}) +\frac{1}{R_{ih}C_{i}} (T_{h}-T_{i}) + \frac{A_{w}\Phi}{C_{i}}  \\
& \qquad +\frac{1}{R_{ie}C_{i}} (T_{e}-T_{i}) + \frac{1}{R_{ia}C_{i}} (T_{a}-T_{i}) 
\end{align*}
\begin{align*}
\frac{\textup{d} T_{h}}{\textup{d}t} & = \frac{1}{R_{ih}C_{h}}(T_{i} - T_{h}) + \frac{u}{C_{h}}\\
\frac{\textup{d} T_{e}}{\textup{d}t} & = \frac{1}{R_{ie}C_{e}}(T_{i} - T_{e}) + \frac{1}{R_{ea}C_{e}}(T_{a} - T_{e}) + \frac{A_{e}\Phi}{C_{e}}
\end{align*}
%\end{subequations}
where the subscripts $ (\cdot)_s,(\cdot)_i,(\cdot)_h, (\cdot)_e $ and $ (\cdot)_{a} $ denotes the sensor, building interior, heater, building envelop, and ambient, respectively. $ T $ denotes the temperature, $ R $ denotes the thermal resistance, $ C $ denotes the heat capacity and $ u $ denotes the heat flux. The solar irradiation $ \Phi $ enters the building interior through the effective window area $ A_{w} $ in addition to heating the building envelop with effective area $ A_{e} $. The states are given by $ x = [T_{s},T_{i},T_{h},T_{e}]^{\mathsf{T}}  $ with $ T_{(\cdot)} \in [12,40] $ $ ^\circ $C. The ambient temperature $ T_{a} \in [-5,20] $ $ ^\circ $C and the solar irradiation $ \Phi \in [0,0.2] $ kW/m$ ^2 $ are measured disturbances. The parameter values used in the model are taken from \cite{bacher2011identifying}. %shown in Table~\ref{tb:BCCparam}. 
The objective is to drive the interior temperature $ T_{i} $ to a desired setpoint $ T_{i}^{sp }\in [18,25]$ $ ^\circ $C , while penalizing the rate of change of the  input usage $ u \in [0,40] $ kW. The stage cost is  given by   \[ \ell(x,u) = (T_{i} - T_{i}^{sp})^2 + 0.1(\Delta u)^2   \] 
The MPC problem is formulated with a sampling time of 1 min and a prediction horizon of $ N= 3$ hours. The goal is to approximate the MPC policy $ \pi_{\textup{mpc}}(\tilde{x}) $.
In this example $ \tilde{x} = [T_{s},T_{i},T_{h},T_{e},T_{i}^{sp},T_{a},\Phi,u]^{\mathsf{T}}$, which requires sampling a 8-dimensional space in order to generate the training samples.   {This example demonstrates the proposed approach when we have a higher-dimensional augmented state-space (cf. Remark~3).}

We randomly generate a total of 6930 samples, out of which, 330 samples were generated by solving the optimization problem  (denoted by  $ \mathcal{D}^0 $) and 6600 samples were generated using the proposed sensitivity update  as shown in Algorithm~\ref{alg:SensitivityTraining} (denoted by  $ \mathcal{D}^+ $). As a benchmark, the same 6600 samples were also generated by solving the full NLP as shown in Algorithm~\ref{alg:Training} (denoted by $ \mathcal{D}^{++} $).

 Using the generated training samples, we approximate the MPC policy using deep neural networks with $ 3  $ hidden layers and $ 10 $ neurons in each hidden layer, with rectified linear units (ReLU) as the activation function in each neuron. We first approximate the policy using only $ \mathcal{D}^0 $. This is denoted by $ \pi_{\textup{approx}}(\tilde{x};\theta_0) $. The policy trained on the  full data set $ \mathcal{D}_{1}:= \mathcal{D}^0 \cup \mathcal{D}^{++}$ is denoted by $ \pi_{\textup{approx}}(\tilde{x};\theta_1) $. Similarly, the policy trained on the proposed sensitivity-based augmented  data set $ \mathcal{D}_{2}:= \mathcal{D}^0 \cup \mathcal{D}^{+}$ is denoted by $ \pi_{\textup{approx}}(\tilde{x};\theta_2) $.

We test the performance of the  approximate MPC policy  for a total simulation time of 12 hours, with changes in the setpoint (at $ t=3 $ h), solar irradiation  (at time $ t=6 $ h), and ambient temperature  ($ t=9 $ h).   Fig.~\ref{Fig:BCCsim} shows the closed loop simulation results using the traditional MPC control law $ \pi_{\textup{mpc}}(\tilde x) $~(blue) obtained by solving the MPC problem online, and the performance of the approximate policies $ \pi_{\textup{approx}}(\tilde{x};\theta_0) $~(gray) , $ \pi_{\textup{approx}}(\tilde{x};\theta_1) $~(green), and $ \pi_{\textup{approx}}(\tilde{x};\theta_2) $~(red). We see $ \pi_{\textup{approx}}(\tilde{x};\theta_0) $ does not approximate the MPC policy accurately, due to too few training samples. On the other hand, $ \pi_{\textup{approx}}(\tilde{x};\theta_1) $  and $ \pi_{\textup{approx}}(\tilde{x};\theta_2) $ mimics the MPC policy closely. However obtaining the training data set for $  \pi_{\textup{approx}}(\tilde{x};\theta_1)  $ is significantly more costly than the proposed  sensitivity-based data augmentation scheme  used to train $  \pi_{\textup{approx}}(\tilde{x};\theta_2)  $.
From this it can be seen that the proposed sensitivity-based data augmentation framework can be used to parameterize the measured disturbances, setpoints, and the control input in addition to the states in order to handle time varying disturbances and setpoints, and approximate the MPC policy closely by augmenting the data set with inexact samples.

\section{CONCLUSIONS} \label{sec:Conclude}
To conclude, this technical note addresses an important implementation aspect of MPC policy approximation, namely the cost of training.  Algorithm~\ref{alg:SensitivityTraining} exploits the parametric sensitivities   to augment  several training samples using the solution of a single optimization problem. 
It was shown that by  using the proposed approach, one can
\begin{itemize}
	\item sample  the  feasible state space sparsely, hence reducing the number of optimization problems that needs to be solved offline,
	\item  and augment the data set with additional samples using a tangential predictor.
\end{itemize}
The error due to augmenting the data set with inexact samples was also quantified, and it was shown to depend quadratically on the  max distance between the augmented sample and the original sample $ \|\Delta x_{i}\|^2 $. 
More broadly, the proposed data augmentation scheme can be used in any policy approximation  where the training data comprises of optimal state-action pairs that is sampled from a policy given by  a nonlinear programming problem. As such, this paper is a first step towards a data augmentation framework for approximate optimal control problems.


%%%%%%%%%%%%%%%%%%%%%%%%%%%%%%%%%%%%%%%%%%%%%%%%%%%%%%%%%%%%%%%%%%%%%%%%%%%%%%%%

%%%%%%%%%%%%%%%%%%%%%%%%%%%%%%%%%%%%%%%%%%%%%%%%%%%%%%%%%%%%%%%%%%%%%%%%%%%%%%%%

%%%%%%%%%%%%%%%%%%%%%%%%%%%%%%%%%%%%%%%%%%%%%%%%%%%%%%%%%%%%%%%%%%%%%%%%%%%%%%%%
%\section*{APPENDIX}

\section*{ACKNOWLEDGMENT}
Helpful discussions with Prof. Sebastien Gros from the Department of Engineering Cybernetics at the Norwegian University of Science and Technology is gratefully acknowledged. 
\bibliographystyle{IEEEtran}
\bibliography{ScenOpt}

%%%%%%%%%%%%%%%%%%%%%%%%%%%%%%%%%%%%%%%%%%%%%%%%%%%%%%%%%%%%%%%%%%%%%%%%%%%%%%%%

\end{document}